\documentclass{amsart}

\usepackage{amsfonts,amscd,amssymb,amsmath,amsthm}
\usepackage{verbatim}
\usepackage{graphicx}
\usepackage[latin1]{inputenc}
\usepackage{tikz}
\usepackage{scalefnt}
\usepackage{caption}
\usepackage{subcaption}
\usepackage{mathrsfs}

\usepackage{xcolor}

\begin{document}

\newtheorem{theorem}{Theorem}[section]
\newtheorem{lemma}[theorem]{Lemma}
\newtheorem{corollary}[theorem]{Corollary}
\newtheorem{conjecture}[theorem]{Conjecture}
\newtheorem{cor}[theorem]{Corollary}
\newtheorem{proposition}[theorem]{Proposition}
\newtheorem{definition}[theorem]{Definition}
\newtheorem{example}[theorem]{Example}
\newtheorem{claim}[theorem]{Claim}
\newtheorem{xca}[theorem]{Exercise}

\theoremstyle{remark}
\newtheorem{remark}[theorem]{Remark}
\bibliographystyle{alpha}

\renewcommand{\P}[1]{\mathbb{P}\left[#1\right]}
\newcommand{\T}{\mathscr{T}}

\title{Warmth and mobility of random graphs}

\author{Sukhada Fadnavis}
\author{Matthew Kahle}
\author{Francisco Martinez-Figueroa}

\date{\today}
\maketitle

\begin{abstract}

A graph homomorphism from the rooted $d$-branching tree $\phi: T^d \to H$ is said to be \emph{cold} if the values of $\phi$ for vertices arbitrarily far away from the root can restrict the value of $\phi$ at the root. \emph{Warmth} is a graph parameter that measures the non-existence of cold maps. We study warmth of random graphs $G(n,p)$, and for every $d \ge 1$, we exhibit a nearly-sharp threshold for the existence of cold maps. As a corollary, for $p=O(n^{-\alpha})$ warmth of $G(n,p)$ is concentrated on at most two values. As another corollary, a conjecture of Lov\'asz relating mobility to chromatic number holds for ``almost all'' graphs.  Finally, our results suggest new conjectures relating graph parameters from statistical physics with graph parameters from equivariant topology.

%

 
%

\end{abstract}

\section{Introduction}

Graph homomorphisms provide a natural combinatorial framework for statistical physics.  In combinatorial statistical physics, one is often interested in an infinite graph $G$ such as a lattice or a tree, representing a ``ground space'', and then a finite target graph $H$, representing physical ``constraints''. Then graph homomorphisms $\phi: G \to H$ correspond to allowable states of the system. See Winkler's ICM notes \cite{Winkler_ICM} for an introduction and overview.

Brightwell and Winkler studied specifically maps from the rooted $d$-branching tree $G=T^d$ to finite and loopless targets (\cite{BW1}). Motivated by statistical physics, they introduced the idea of a ``cold map''.  A cold map is a graph homomorphism $\phi: T^d \to H$ where the value of $\phi$ at vertices arbitrarily far away from the root can still affect what values are possible at the root---a precise definition is given below. The \emph{warmth} of a graph $H$, denoted $w(H)$, is the maximal $d$ such that there do no exist any cold maps $\phi: T^{d-2} \to H$. It turns out that $w(K_n)=n$ and $w(H)=2$ for any bipartite graph. 

Our main interest here is the warmth of the edge-independent random graph $G(n,p)$. For every $d \ge 1$, we find a nearly sharp phase transition for the existence of cold maps $T^d \to G(n,p)$, analogous to a solid-liquid phase transition. Our results are illustrated in Figure \ref{fig:summary}.
As a corollary of our results, warmth is concentrated at most at two values, as long as $p=O \left(n^{-\alpha} \right)$. This is perhaps surprising, as warmth is not a monotone property.

Brightwell and Winkler proved a number of results relating statistical physics with graph coloring. As one corollary of our main results, we verify that a conjecture of Lov\'asz that \emph{mobility} (defined in Section \ref{sec:mobile}) is a lower bound on chromatic number holds ``almost always,'' in the sense that it holds for random graphs with high probability.

Methods from equivariant algebraic topology have also provided lower bounds on chromatic number, for example in work of Lov\'asz \cite{Lovasz}, and Babson and Kozlov \cite{BK1,BK2}. Our results here, combined with some earlier work \cite{Kahle_neighborhood}, suggest a strong correlation between some of these parameters for random graphs, thus we make deterministic conjectures to relate the statistical physics and topological sides. Namely, we make Conjecture \ref{conj:conn} that
$$ w(H) \le \mbox{conn}[\mathcal{N}(H)] + 3,$$
for every finite graph $H$. Partly motivated by an earlier version of this paper, Dochtermann and Freij-Hollanti \cite{DFH18} proved this conjecture in some special cases.

\begin{figure}
\centering
\scalefont{2}
\scalebox{0.75}{
\begin{tikzpicture}[thick]
\large

\draw [fill=blue] (1, 1.5) rectangle (7,  2); 
\draw [fill=blue] (1, 2) rectangle (10, 2.5);
\draw [fill=blue] (1, 2.5) rectangle (11, 3);
\draw [fill=blue] (1, 3 ) rectangle (11.5, 3.5);
\draw [fill=blue] (1, 3.5) rectangle (11.8, 4);
\draw [fill=blue] (1, 4) rectangle (12, 4.5);
\draw [fill=blue] (1, 4.5) rectangle (12.1428, 5);
\draw [fill=blue] (1, 5) rectangle (12.25, 5.5);
\draw [fill=blue] (1, 5.5) rectangle (12.33, 6);
\draw [fill=blue] (1, 6) rectangle (12.4, 6.5);

\draw (1,1.4) -- (1,1.6);
\draw    (7,1.4) -- (7,1.6);
\draw (10,1.4) -- (10,1.6); 
\draw (11,1.4) -- (11,1.6);
\draw (11.5,1.4) -- (11.5,1.6);
\draw (11.8,1.4) -- (11.8,1.6);
\draw (11.8,1.4) -- (11.8,1.6);
\draw (12,1.4) -- (12,1.6);
\draw(12.1429,1.4) -- (12.1429,1.6);
\draw (12.25,1.4) -- (12.25,1.6);
\draw (12.33,1.4) -- (12.33,1.6);
\draw (12.4,1.4) -- (12.4,1.6);
\draw (12.4545,1.4) -- (12.4545,1.6);

\node at (0.85,1) {\bf $-2$};
\node at (6.85,1) {\bf $-1$};
\node at (9.85,1) {\bf $- \frac{1}{2}$};
\node at (10.85,1) {\bf $- \frac{1}{3}$};
\node at (13,1) {\bf $0$};

\node at (0.5,1.75) {$1$};
\node at (0.5,2.25) {$2$};
\node at (0.5,2.75) {$3$};
\node at (0.5,3.25) {$4$};
\node at (0.5,3.75) {$5$};
\node at (0.5,4.25) {$6$};
\node at (0.5,4.75) {$7$};
\node at (0.5,5.25) {$8$};
\node at (0.5,5.75) {$9$};
\node at (0.5,6.25) {$10$};
\node at (13.75,1.5) {\bf \huge $\alpha$};
\node at (1,7.25) {\bf \huge $d$};
\draw (1,1.5) -- (13,1.5);
\draw (1,1.5) -- (1,6.5);
\draw (13,1.4) -- (13,6.5);
\end{tikzpicture}
}
\caption{A summary of our results. The horizontal axis is $\alpha$, where $p = n^{\alpha}$ is the edge probability. The vertical axis $d$, represents $d$-branching rooted tree $T^d$. In the blue region, cold maps $T^d \to G(n,p)$ exist with high probability; in the white region, they do not.}
\label{fig:summary}
\end{figure}

\section{Background}



All the graphs we consider here are simple graphs, meaning undirected and without multiple edges or loops. For a graph $G$ with vertices $u, v \in V(G)$, we write $\{u,v\}\in E(G)$ to denote adjacency. Other than infinite branching trees $T^d$, we will consider only finite graphs. 

\begin{definition} For graphs $G$ and $H$, a function $\phi: V(G) \rightarrow V(H)$ is a \emph{graph homomorphism} if it respects adjacency. That is, whenever $\{u,v\}\in E(G)$, we have $\{\phi(u),\phi(v)\}\in E(H)$.
\end{definition}

Graph homomorphisms generalize graph colorings, since an $n$-coloring of $H$ is equivalent to a homomorphism $H \to K_n$. For an overview of graph homomorphisms, see the book \cite{Hell}.

The set of all homomorphisms of a graph $G$ to a graph $H$ is denoted Hom($G,H$). It
is naturally endowed with a graph structure as follows: the elements of Hom($G,H$) are the vertices of a graph, with an edge between
two homomorphisms whenever they differ in exactly one vertex. There is also a natural higher-dimensional structure on Hom($G$,$H$) which is important in
topological combinatorics \cite{BK1}. \\

Let $T^d$ denote the (infinite) $d$-branching tree, as illustrated in Figure \ref{fig:tree}. Note that every vertex has degree $d+1$ except the root, which has degree $d$.


\tikzstyle{every node}=[circle, draw, fill=blue,
                        inner sep=0pt, minimum width=4pt]
\begin{figure}
\begin{tikzpicture}[scale=0.8,line width=1pt]
\foreach \x in {0, ...,3}
{
\draw (0,0)--(\x*120:1);
}
\foreach \x in {0, ..., 2}
{
\draw (\x*120:1)--(\x*120-40:2);
\draw (\x*120:1)--(\x*120:2);
\draw (\x*120:1)--(\x*120+40:2);
}
\foreach \x in {0, ..., 8}
{
\draw (\x*40:2)--(\x*40-13.3333:3);
\draw (\x*40:2)--(\x*40:3);
\draw (\x*40:2)--(\x*40+13.3333:3);
}
\foreach \x in {0,...,26}
{
\foreach \y in {-1,...,1}
{
\draw (\x*13.3333:3)--(\x*13.3333+\y*4.44444:4);
}
}

\foreach \y in {0,...,80}
{
\draw (\y*4.4444:4) node{};
}
\foreach \y in {0,...,26}
{
\draw (\y*13.3333:3) node{};
}
\foreach \y in {0,...,8}
{
\draw (\y*40:2) node{};
}
\foreach \y in {0,...,2}
{
\draw (\y*120:1) node{};
}
\draw (0,0) node{};

\end{tikzpicture}
\caption{The (truncated) $3$-branching tree $T^3$.}
\label{fig:tree}
\end{figure}

\begin{definition}\label{def:cold_and_warm} A map $\varphi$ in Hom($T^d,H$) is said to be \emph{cold} if there is a node $a$ of $H$ such that for any $k\geq 0$, whenever a map
$\psi \in $ Hom($T^d,H$) agrees with $\varphi$ on the sites at distance $k$ from the root $r$, then $\psi(r)\neq a$. A graph $H$ is said to be
\emph{$d$-warm} if Hom($T^{d-2}, H$) does not contain any cold maps. Furthermore, the \emph{warmth} $w(H)$ of $H$ is defined to be the largest $d$
for which $H$ is $d$-warm.
\end{definition}

For a graph $H$ and a vertex $u \in H$ let $N(u)$ denote the set of vertices in $H$ adjacent to $u$. $N(u)$ is called the \emph{neighborhood}
of $u$. For a set of vertices $A$, define the neighborhood $N(A) = \cup_{a \in A} N(a)$.

We state a theorem of Brightwell and Winkler, which gives an equivalent characterization of warmth.

\begin{definition}   A family of subsets $\{A_i\}_{1}^{w}$ of $H$ is called a \emph{d-stable} family if for all $1 \leq i \leq w$ there
exist $A_{i_{1}} \cdots A_{i_{d}}$ such that $$\bigcap_{j=1}^{d} N(A_{i_j}) = A_i.$$
\end{definition}

\begin{theorem}[Brightwell and Winkler \cite{BW1}]\label{thm:stable_fam} \label{thm1}  Given a constraint graph $H$ and $d \geq 1$, the following are equivalent:

\begin{itemize}
\item $H$ is not $(d+2)$-warm; i.e.\ there exists a cold map $\phi: T^d \to H$.
\item There exists a $d$-stable family of subsets of $H$.\\
\end{itemize}
\end{theorem}

\section{Statement of Results}

In the following sections we study the warmth of Erd\H{o}s-R\'enyi random graphs and show that warmth is much smaller than the chromatic number.  By the Brightwell-Winkler inequalities, the mobility is also relatively
small, and comparing with known results about chromatic numbers of random graphs gives Conjecture \ref{LC}  for ``almost all'' graphs.

The following is a summary of results.  We use Bachmann--Landau and related notations: $O, o, \Omega, \omega, \Theta$.  In every case, the
asymptotic notation is to be understood as the number of vertices $n \to \infty$. We say that an event happens \emph{asymptotically almost surely (a.a.s)} if its probability approaches $1$ as $n \to \infty$.

The statement of the theorems and proofs are similar for the sparse and dense cases -- however it eases notation and simplifies the proofs to treat the sparse and dense cases separately. 

\begin{theorem}[Sparse regime]\label{sparse} For a constant $\alpha>0$:\leavevmode 

If $p=O( n^{-\alpha})$, then a.a.s. $$w(G(n,p)) \le \lfloor 1/ \alpha +2 \rfloor.$$

If $p = \Omega ( n^{-\alpha})$, then a.a.s. $$w(G(n,p)) \ge \lceil 1/ \alpha + 1 \rceil .$$
\end{theorem}
As a corollary we have that in the regime $p=O \left(n^{-\alpha} \right)$, where $\alpha >0$ is fixed, the warmth is concentrated on at most two values and that for ``most''
sequences $p=p(n)$ is concentrated on one value.  For instance if $p = \Theta (n^{-\alpha})$ with $1/(k+1) < \alpha < 1/k$ then a.a.s.\ we have
$w = k+2$.  See Figure \ref{fig:summary}.

We also have the following bounds on warmth in the dense regime.

\begin{theorem}[Dense regime]\label{dense}
If  $p=1/2$, then a.a.s%
$$\log_2{n}-2\log_2{\log_2{n}} \leq w(G(n,p)) \le \log_2 {n} +2.$$
\end{theorem}



In Section \ref{sec:mobile}, we apply our main results to show that a conjecture of Lov\'asz holds for almost all graphs. In Section \ref{sec:topology}, we dicuss possible connections between statistical-physics and topological parameters that each give lower bounds on chromatic number.





\section{Upper bounds}
The main idea of the proof is to construct $s$-stable families and invoke Theorem \ref{thm:stable_fam} to show warmth is at most $s+1$. In both the sparse and dense cases, the $s$-stable families can be obtained with high probability by the collection of singletons of vertices. This reduces to showing that the probability all $s$-sets of neighbors of a vertex $v$ share another common neighbor, goes to 0 as $n\to\infty$. We achieve this by considering disjoint $s$-subsets of $N(v)$ and using the independence of events to bound this probability.

\begin{proof} [Proof of Upper Bounds in Theorems \ref{sparse} and \ref{dense}.] Let $\delta(H)$ denote the minimum degree of $H$. It will be convenient to assume $\delta(G(n,p))$ is big, so we treat the case $p=O(n^{-0.99})$ separately. If $G(n,p)$ is disconnected, ${w(G(n,p))\leq 2}$, so we may assume $p$ is over the connectivity threshold. Let $u,v,w$ be a triangle of $G(n,p)$ (which exists a.a.s.). The probability that $\{\{u\},\{v\},\{w\}\}$ is not a 2-stable family is the same as the probability that there is a vertex $z$ connected to two of these vertices at the same time, which is bounded above by $3np^2=o(1)$. Then, a.a.s. it is a 2-stable family, and by Theorem \ref{thm:stable_fam}, $w(G(n,p))\leq 3$. 

We now assume $$p\geq\frac{\log n+\omega(\log\log n)}{n},$$ which ensures that for any fixed $k$, $\delta(G(n,p))\geq k$ a.a.s. \cite{Bollobas}. In the sparse case, let $p\leq C n^{-\alpha}$ for some constants $C>0, 0<\alpha\leq 0.99$; in this case set $s= \lfloor 1/\alpha + 1\rfloor$.  In the dense case, we assume that $p=1/2$, and  set $s = \lfloor\log_{2}{n}+1\rfloor $.  In both cases we will show that $$w(G(n,p)) \leq s+1.$$

Let $c=\frac{2s}{s\alpha-1}$ in the sparse case and $c(n) = n/4$ in the dense case. Observe $\P{\delta(G(n,p)) \leq c} \rightarrow 0$ in both cases \cite{Bollobas}, so we shall assume that $\delta(G) > c$, for $G=G(n,p)$.

Consider the family $\mathcal{F} = \{\{v\} \mid v \in V \}$ of singleton vertices of $G$. We show that a.a.s.\ $\mathcal{F}$ is an $s$-stable family, i.e.\ for every vertex $v$ there exist $\{ u_1, \cdots , u_s \} \subseteq N(v)$ such that $$\bigcap_{i=1}^s N(u_i)= \{v\}.$$ We call such a set $\{u_1, \ldots, u_s\}$ an \emph{$s$-representative} of $v$ in $\mathcal{F}$.

For $u_1, \cdots, u_s$ neighbors of $v$, the probability that they are also adjacent to vertex $u$  is $p^s$ by edge independence. Let $A(w)(w_1, \cdots, w_s)$ denote the event that $w_1, \cdots, w_s$ are in the neighborhood $N(w)$ of $w$. Then,
$$\P{A(u)(u_1, \cdots, u_s) \text{ for some } u \in V} \leq \sum_{u \in V} \P{A(u)(u_1, \cdots, u_s)} = np^s,$$
where $\P{A(u)(u_1, \cdots, u_s) \text{ for some } u \in V}$ is the probability that $v$ shares $u_1, \cdots, u_s$ as neighbors with some other vertex in $V$.
Let $\gamma = \lfloor c/s \rfloor$. We consider $N_1, \cdots, N_{\gamma}$  some disjoint subsets of the neighbors of $v$, such that $|N_i| = s$. We can do so by the assumption that $|N(v)| \geq c$. Let $$N(v,s) = \{U\subset N(v), |U| = s\}.$$ For $M \in N(v,s)$ let $A(M)$ denote the event that $M \subseteq N(u)$ for some $u \in V$.
We have
\begin{equation*}
    \P{\bigcap_{M \in N(v,s)} A(M)} \leq \P{\bigcap_{i = 1}^{\gamma} A(N_i)}.
\end{equation*}
That is, the probability that $v$ shares every $s$-subset of $N(v)$ with some other vertex of $V$ is less than the probability that $v$ shares each of the $N_i$ as neighbors with some other vertex in $V$.

Since the $N_i$ are disjoint, we have, $$\P{\bigcap_{i = 1}^{\gamma} A(N_i)} = (np^s)^{\gamma} .$$
Thus, the probability that $G(n,p)$ does not have a $s$-stable family is less than the probability that some vertex does not have an $s$-representative in $\mathcal{F}$, which is
$$O\left(\sum_{v \in V} (np^s)^{\gamma}\right) = O(n(np^s)^{\gamma}).$$
In the sparse case $\gamma(1-s\alpha)\leq\frac{c}{s}(1-s\alpha)\leq -2$. Then, the probability that $G$ does not have an $s$-stable family is 
$$O\left(n(np^s)^\gamma\right)=O\left(n^{1+\gamma(1-s\alpha)}\right)= O\left(n^{-1}\right)=o(1).$$

In the dense case, $$n(np^s)^\gamma=n^{1+\gamma}p^{s\gamma}=2^{(1+\gamma)\log_2{n}-s\gamma}=2^{-[\gamma(s-\log_2{n})-\log_2{n}]}$$

Here $\log_2{n}+1/2\leq s\leq 2\log_2{n}$, so for $n$ large, there exist $\epsilon>0$ such that
\begin{align*}
    \gamma(s-\log_2n)-\log_2n &\geq \frac{n}{5s}(s-\log_2n)-\log_2n \geq \frac{n}{10s}-\log_2n\\
    &\geq \frac{n}{20\log_2n}-\log_2n=\frac{n-20\left(\log_2 n\right)^2}{20\log_2n}\geq n^{\epsilon}.
\end{align*}

Thus the probability that $G$ does not have an $s$-stable family is 
$$O(2^{-[\gamma(s-\log_2{n})-\log_2{n}]})=O(2^{-n^\epsilon})=o(1). $$

This proves that in both cases a.a.s.\ $G(n,p)$ has an $s$-stable family of subsets and hence by Theorem \ref{thm1}, a.a.s.\ $w(G(n,p))
\leq s+1$. 
\end{proof}

\section{Lower bounds} \label{sec:lower}

We now discuss maps from the infinite $s$-branching tree $T^s$ and it is convenient to label its vertices with the root labeled $1$ and its children labeled $2, 3, 4 \dots$, as shown in Figure \ref{fig:three_branching_map}. We will also consider $T^s_\ell$, the finite $s$-branching tree of length $\ell$,  to be the induced subgraph of $T^s$ by all vertices within distance $\ell$ of the root. 


\begin{figure}
\begin{tikzpicture}[line width=1pt]
\foreach \x in {0, ...,3}
{
\draw (0,0)--(\x*120:1);
}
\foreach \x in {0, ..., 2}
{
\draw (\x*120:1)--(\x*120-40:2);
\draw (\x*120:1)--(\x*120:2);
\draw (\x*120:1)--(\x*120+40:2);
}
\foreach \x in {0, ..., 8}
{
\draw (\x*40:2)--(\x*40-13.3333:3);
\draw (\x*40:2)--(\x*40:3);
\draw (\x*40:2)--(\x*40+13.3333:3);
}

\foreach \y in {0,...,26}
{
\draw (\y*13.3333:3) node{};
}

\foreach \y in {0,...,8}
{
\draw (\y*40:2) node{};
}

\foreach \y in {0,...,2}
{
\draw (\y*120:1) node{};
}

\draw (0,0) node{};
\tikzstyle{every node}=[]
\draw (0.2,0.2) node{$1$};

\foreach \x in {2, ..., 4}
{
\draw (120*\x-135:0.9) node{$\x$};
}

\foreach \x in {5,...,13}
{
\draw (\x*40-208:2) node{$\x$};
}

\foreach \x in {14,...,40}
{
\draw ((\x*13.3333-240:3.3) node{$\x$};
}
\end{tikzpicture}
\caption{The labeling of vertices for $T^3_3$.}
\label{fig:three_branching_map}
\end{figure}

Given $s$ and $\ell$, set $m=1+s+\dots+s^{\ell-1}=\frac{s^\ell-1}{s-1}$.  Then $T^s_\ell$ has $$1+s+s^2+\dots+s^\ell=1+s(1+\dots+s^{\ell-1})=1+sm$$ vertices, and $s^\ell$ leaves $L=\{m+1,m+2,\dots, m+s^\ell\}$. 

Let  $D = D (T^{s}_\ell)$ denote the set of leaves of $T^{s}_\ell$, together with the root; i.e. $$D = \{ 1 \} \cup L. $$ Let $H$ be any graph, and
$f :   D \to V(H)$ any function.  We say that $f$ is {\it extendable} if there exists a graph homomorphism $\phi : T^{s}_\ell \to H$ such that
$\phi |_{D} = f$, and that $H$ has {\it property $\mathcal{P}^s_{\ell}$} if every function $f: D( T^{s}_\ell) \to V(H)$ is extendable.

Property $\mathcal{P}^s_{\ell}$ precludes any cold maps $T^s \to H$, so by definition such graphs have warmth at least $s+2$. A useful observation is that property $\mathcal{P}^s_{\ell}$ is monotone, even though warmth is not. Thus if property $\mathcal{P}^s_{\ell}$ holds for $G(n,p)$ a.a.s.\
then it also holds for $G(n,p')$ a.a.s.\ whenever $p \leq p'$.

In the following couple of lemmas, we start by bounding the probability that $G(n,p)$ has property $\mathcal{P}^s_{\ell}$. To do so we apply Janson's inequality (see \cite{Alon}). 

\begin{lemma}\label{lemma:extendable_map}
Let $s,\ell\geq 2$ be constant integers and $p=n^{-\alpha}$. Suppose $$1/m< 1-s\alpha <1/(m-1),$$ where $m$ is as defined above. 
Given fixed vertices $w_1,w_{m+1}, w_{m+2},\dots, w_{m+s^\ell}$ in $[n]$, define $f:D\to G(n,p)$ by $f(i)=w_i$ for $i\in D(T^s_\ell)$. Then there exist positive constants $C$ and $\epsilon$ such that 
$$\P{f \textup{ is not extendable in } G(n,p)}\leq e^{-Cn^\epsilon}.$$
\end{lemma}

\begin{proof}
Let $W=\{w_1,w_{m+1},\dots, w_{m+s^\ell}\}$ and $\epsilon=\frac{1-s\alpha}{s}$.
Let $\mathcal{A}$ be the collection of ordered $(m-1)$-subsets of $[n]\setminus W$, thus $$(m-1)!\binom{n-|W|}{m-1}\leq |\mathcal{A}|,$$ since $\ell$, $s$ and $m$ are constants, $|A|=\Omega(n^{m-1})$.

Given $A_i=\{w^i_2,\dots,w^i_m\}\in\mathcal{A}$ (in that order), define $\phi_i:T^s_\ell\to [n]$ by

\begin{equation}\label{map_phi}
\phi_{i}(j) = \left\{
\begin{array}{rl}
w_1 & \text{if }  j=1\\
w_j^i & \text{if } 2\leq j \leq m \\
w_j & \text{if } m+1 \leq j \leq m+s^\ell
\end{array} \right.
\end{equation}

Let $A'_i=\phi_i(T^s_\ell)$, and let $B_i$ be the event that $A_i'$ is a subgraph of $G(n,p)$. Note that the event $B_i$ is equivalent to $\phi_i$ being a graph homomorphism $\phi_i:T^s_\ell\to G(n,p)$. 

Also note $|E(A_i')|=|E(T^s_\ell)|=ms$, since $\phi_i$ is edge-injective. Thus $\P{B_i}=p^{ms}$. We will bound $$\P{f \text{ is not extendable}}\leq\P{\bigwedge_{i=1}^{|\mathcal{A}|}\overline{B_i}}.$$

Since in general the events $B_i$ are not independent, we use Janson's inequality to bound this probability. See \cite{Alon} for more on this inequality. 

Write $i\sim j$ whenever $i\neq j$ and $E(A'_i)\cap E(A'_j)\neq\emptyset$. Note that $B_i$ and $B_j$ are not independent if and only if $i\sim j$. 

Since $\ell\geq 2$, whenever $A'_i$ has an edge in common with $A'_j$, it must be that $A_i$ and $A_j$ share at least one vertex. If $k=|A_i\cap A_j|$, then $1\leq k\leq m-1$, and%
$$|E(A'_i\cup A'_j)|=2ms-|E(A'_i\cap A'_j)|\geq 2ms-ks,$$%
since $A_i'$ and $A_j'$ can share at most $ks$ edges, when all the $k$ vertices they share are parents of leaves. Thus
\begin{align*}
    \P{B_i\wedge B_j}=&\P{A'_i\leq G(n,p) \text{ and } A'_j\leq G(n,p)}\\
    =&\P{(A'_i\cup A'_j)\leq G(n,p)}=p^{|E(A'_i\cup A'_j)|}\leq p^{2ms-ks}.
\end{align*}

Let $\mathcal{A}_k\subset\mathcal{A}\times\mathcal{A}$ be the set of pairs $(A_i,A_j)$ such that $i\sim j$ and $|A_i\cap A_j|=k$. Note there is a constant $C$ depending only on $s$ and $\ell$, such that $$|\mathcal{A}_k|\leq C\binom{n-|W|}{2m-2-k}\leq Cn^{2m-2-k},$$
since $|A_i\cup A_j|=2m-2-k$. 

We compute now $\mu$ and $\Delta$ as in \cite{Alon} to apply Janson's inequality. 
\begin{align*}
    \mu &=\sum_{A_i\in\mathcal{A}}\P{B_i}=|\mathcal{A}|\P{B_1}=\Omega\left(n^{m-1}\right)p^{sm}=\Omega\left( (np^s)^mn^{-1}\right).\\
    \Delta &=\sum_{i\sim j}\P{B_i\wedge B_j}\leq \sum_{k=1}^{m-1}|\mathcal{A}_k|p^{2ms-ks}\leq \sum_{k=1}^{m-1} Cn^{2m-2-k}p^{2ms-ks}\\
    &= C n^{2m-2}p^{2ms}\sum_{k=1}^{m-1} (np^s)^{-k}= C n^{2m-3}p^{2ms-s}\sum_{k=0}^{m-2}\left(\frac{1}{np^s}\right)^k\\
    &= O\left( (np^s)^{2m-1}n^{-2}\right).
\end{align*}

Where the last inequality follows since for $n$ large enough, $np^s=n^{1-s\alpha}>2$ so the sum is bounded above by 2.

Hence we have $$\frac{\Delta}{\mu}\leq O\left(\frac{(np^s)^{2m-1}n^{-2}}{(np^s)^mn^{-1}}\right)= O\left((np^s)^{m-1}n^{-1}\right)=O\left(n^{(1-s\alpha)(m-1)-1}\right)=o(1),$$
since $1-s\alpha<1/(m-1)$. 

Finally, let $\epsilon=(1-s\alpha)m-1$, so $\epsilon>0$ by hypothesis and $\mu> C n^\epsilon$, for some constant $C>0$. Applying Janson's Inequality we get%
$$\P{\bigwedge_i \overline{B_i}}\leq e^{-\mu+\Delta/2}=e^{-\mu(1-o(1))}\leq e^{-\mu/2} \leq e^{-Cn^\epsilon}.$$
\end{proof}

\begin{corollary}\label{lemma:property_s_ell_dense}
For $s,\ell$ and $p$ as in the previous lemma, a.a.s. $G(n,p)$ has property $\mathcal{P}^s_{\ell}$.
\end{corollary}

\begin{proof}
There are less than $n^{1+s^\ell}$ possible sets $W$, thus
\begin{align*}
    \P{\text{$f$ corresponding to some $W$ is not extendable}}&\leq n^{1+s^\ell}e^{-Cn^\epsilon}\\
    =\exp{\left((1+s^\ell)\log{n} -Cn^\epsilon\right)}&=o(1).
\end{align*}
\end{proof}

We now complete the proof. 

\begin{proof}[Proof of Lower Bound in Theorem \ref{sparse}]
Trivially $w(G(n,p))\geq 2$ if it has any edges, so we focus on the non-trivial cases and assume $\alpha<1$. We prove that for $p=\Omega(n^{-\alpha})$ and $s=\lceil1/\alpha -1\rceil$, there exists an $\ell>2$ such that a.a.s. $G(n,p)$ has property $\mathcal{P}^s_{\ell}$, so $w(G(n,p))\geq s+2$. 

Note that $1-s\alpha>0$, so there must exist some $\ell>2$ such that $$\frac{1}{m}=\frac{s-1}{s^\ell-1}<1-s\alpha.$$
Thus, there exists $\beta\geq\alpha$ such that $$1/m<1-s\beta<1/(m-1).$$
Applying lemma \ref{lemma:extendable_map} and its corollary, a.a.s. $G(n,p^{-\beta})$ has property $\mathcal{P}^s_{\ell}$. Since $n^{-{\beta}}\leq n^{-\alpha}$, by monotonicity, a.a.s. $G(n,p)$ also has property $\mathcal{P}^s_{\ell}$, as desired. 
\end{proof}

We say that a graph $G$ is \textit{$s-$neighborly} if every $s$-subset of its vertices has a common neighbor. Note that $G$ is $(s+1)$-neighborly if and only if $G$ has property $\mathcal{P}^s_{2}$, and in such case $w(G)\geq s+2$. With this in mind, we now prove the lower bound in Theorem \ref{dense}.

\begin{proof}[Proof of Lower Bound in Theorem \ref{dense}]
Given a $s$-subset $S\subset [n]$, the probability that it has no common neighbors is 
$$\P{\bigcap_{v\in S}N(v)=\emptyset}=(1-p^s)^{n-s}.$$
Thus $\P{G(n,p) \text{ is not $s$-neighborly}}\leq\binom{n}{s}(1-p^s)^{n-s}\leq n^se^{-p^s(n-s)}$. 

Let $p=1/2$ and $s=\lfloor\log_2{n}-2\log_2\log_2{n}\rfloor\leq \log_2n-2\log_2\log_2n\leq\log_2n$. So 

\begin{align*}
n^s &\leq n^{\log_2n}=e^{\ln{2}(\log_2n)^2},\\
n-s &\geq n-\log_2n \text{, and}\\ 
2^s &\leq 2^{(\log_2n-2\log_2\log_2n)}
\end{align*}
Hence\begin{align*}
    -p^s(n-s)&=-2^{-s}(n-s)\leq -2^{(2\log_2\log_2n-\log_2n)}(n-\log_2n)\\
    &\leq -2^{(2\log_2\log_2n)}\left(1-\frac{\log_2n}{n}\right)=-(\log_2n)^2\left(1-\frac{\log_2n}{n}\right).
\end{align*}
So
\begin{align*}
    \P{G(n,1/2) \text{ is not $s$-neighborly}}&\leq n^se^{-p^s(n-s)}\leq \exp\left(\ln2(\log_2n)^2-(\log_2n)^2\left(1-\frac{\log_2n}{n}\right)\right)\\
    &= \exp\left(-(\log_2n)^2\left[1-\ln2-\frac{\log_2n}{n}\right]\right)\to0\text{   as }n\to 0.
\end{align*}
Thus a.a.s. $G(n,1/2)$ is $s$-neighborly and $$w(G(n,1/2))\geq \log_2n-2\log_2\log_2n+2.\eqno\qedhere$$
\end{proof}

\section{Mobility and Lov\'asz's conjecture} \label{sec:mobile}

Along with the definition of warmth, in \cite{BW1} Brightwell and Winkler also defined the mobility of a graph, which instead of focusing on graph homomorphisms coming from branching trees, describes the connectivity of Hom($H$,$G$) for general ``test graphs'' $H$ of a given maximum degree. While Brightwell and Winkler themselves proved inequalities relating these invariants, Lov\'asz conjectured an extra inequality between mobility and the chromatic number of a graph (Conjecture \ref{LC}). Below, we recall these definitions and results, which combined with our results for the warmth of random graphs, will show that Lov\'asz's conjecture is true a.a.s. for random graphs. 

\begin{definition} Let $\mathcal{H}_d$ denote the set of all finite simple graphs of maximum degree $d$. The graph $G$ is said to be \emph{$d$-mobile} if for each $H\in\mathcal{H}_{d-2}$, Hom($H, G$) is connected or empty.
The largest $d$ for which $G$ is $d$-mobile is the \emph{mobility}, $m(G)$, of $G$.
\end{definition}

Brightwell and Winkler established the following inequalities relating warmth, mobility, and chromatic number.

\begin{theorem}[Brightwell and Winkler \cite{BW1}] \label{ineqs} For every finite unlooped graph $G$, we have

\begin{itemize}
\item $w(G) \leq \chi(G)$,
\item $m(G) \leq 2w(G) -1$, and therefore
\item $m(G) \leq 2\chi(G) -1$.\\
\end{itemize}
\end{theorem}

A conjecture attributed to Lov\'asz in \cite{BW1} is that the last inequality can be improved to the following.

\begin{conjecture} \label{LC}  For all finite unlooped graphs $G$,
$$ m(G) \leq \chi(G). $$
\end{conjecture}

Brightwell and Winkler have shown this for $\chi(H) \le 3$, and Lov\'asz may have an unpublished proof for $\chi(H) \le 4$ \cite{BW1},
but in general the conjecture is open.  One of the contributions of this article is to establish the conjecture for many graphs (in some sense, ``most graphs'') with large chromatic number. We now prove the following, that establishes Conjecture \ref{LC} a.a.s.\ for random graphs. 

\begin{corollary}\label{coro:LC}
If $p=O(n^{-\alpha})$, for some $\alpha>0$, then a.a.s. $$m(G(n,p))<\chi(G(n,p)).$$
\end{corollary}

\begin{proof}[Proof of Corollary \ref{coro:LC}]\leavevmode
\begin{enumerate}
    \item When $p=\omega\left(1/n\right)$, it is known that a.a.s. $\chi(G(n,p))\to\infty$ as $n\to\infty$ (e.g. see Theorem 11.29 in \cite{Bollobas1}). In contrast, by  Theorem \ref{sparse}, a.a.s. $w(G(n,p))\leq\lfloor1/\alpha+2\rfloor$, so by Theorem \ref{ineqs}, $m(G(n,p))= O(1)$. 
    \item When $p=O(1/n)$, Theorem $\ref{sparse}$ implies that a.a.s $w(G(n,p))\leq 3$, so Theorem \ref{ineqs} produces the bound $m(G(n,p))\leq 5$. If $\chi(G(n,p))\geq5$ we are done. Otherwise, $\chi(G(n,p))\leq 4$ and these cases have already been established in \cite{BW1}.\qedhere
    \end{enumerate}
\end{proof}

\section{Topological bounds on chromatic number, Kneser graphs, and connections} \label{sec:topology}

Another approach to lower bounding chromatic number of a graph was pioneered by Lov\'asz \cite{Lovasz} and further developed, for example, in \cite{BK1,BK2, Kozlov}. See Matousek's book \cite{Matousek} for an introduction and overview of equivariant topological methods in combinatorics.

The neighborhood complex of a graph was introduced by L\'ovasz in \cite{Lovasz} as a tool to prove Kneser's conjecture. We state some useful definitions and results from \cite{Lovasz}, but we skip the more technical details. For a more comprehensive exposition on simplicial complexes and topology in general, see \cite{Kozlov, Hatcher}.

\begin{definition}\label{def:neigh_compmlex}
Given a graph $G$, its neighborhood complex $\mathcal{N}(G)$ is the simplicial complex on the vertices of $G$ with simplices all $A\subset V(G)$ such that $A$ has a common neighbor, i.e. $A$ spans a face if and only if $\cap_{v\in A} N(v)\neq\emptyset$. 
\end{definition}

Lov\'asz proved that the connectivity of $\mathcal{N}(G)$ (in the sense of homotopy theory, see \cite{Hatcher}) bounds the chromatic number as follows:

\begin{theorem}[Lov\'asz \cite{Lovasz}]
Let $G$ be a finite simple graph. Then $$\textup{conn}\left[\mathcal{N}(G)\right]+3\leq \chi(G).$$
\end{theorem}

The topological approach shares some similarities with the statistical physics approach. In both settings, one probes the target graph $H$ with maps from a ``test graph'' $T$, often with some symmetry. Both settings give tight lower bounds on chromatic number for certain infinite families of graphs, including bipartite graphs and complete graphs. 
That is,
$$\chi(H) = w(H) =\textup{conn}\left[\mathcal{N}(G)\right]+3$$
for all of these examples. 



The neighborhood complex of a graph is more related to the warmth of a graph that it may appear at first sight. On one hand, the warmth depends on the structure of Hom($T^s$, $G$) for infinite branching trees $T^s$. On the other hand, Babson and Kozlov (\cite{BK2}) discussed a general cell complex structure for Hom($H$,$G$), and proved that the neighborhood complex $\mathcal{N}(G)$ is homotopy equivalent to Hom($T$, $G$) for any finite tree $T$. Moreover, further evidence of a possible relationship comes from the work of Kahle \cite{Kahle_neighborhood} computing the connectivity of the neighborhood complex of random graphs. In summary, he proved that a.a.s.%
\begin{align*}
    \frac{1}{\alpha}-3 &\leq \text{conn}\left[\mathcal{N}(G(n,n^{-\alpha})\right] \leq \frac{4}{\alpha}-3\ \text{ and}\\
    (1-\epsilon)\log_2{n} &\leq \text{conn}\left[\mathcal{N}(G(n,1/2)\right]\leq (4+\epsilon)\log_2{n},
\end{align*}
which are highly correlated to our results in Theorems \ref{sparse} and \ref{dense}. 

To further explore possible connections between the parameters, we consider the example of Kneser graphs. 

\begin{definition}
For $n \geq 2k$ the Kneser graph $KG_{n,k}$ has $\binom{n}{k}$ vertices, one for each $k$-subset of an $n$-set, with edges between disjoint subsets.
\end{definition}

As part of his proof of Kneser's conjecture \cite{Lovasz}, Lov\'asz computed the connectivity of their neighborhood complex. 

\begin{theorem}[\cite{Lovasz}]\label{thm:kneser_conn}
For $n\geq 2k$, $$\textup{conn}\left[\mathcal{N}(KG_{n,k})\right]=n-2k-1.$$ 
\end{theorem}

In the following theorem, we compute their warmth. 

\begin{theorem}\label{thm:kneser}
Let $k\geq 1$ and $n\geq 2k$, then 
$$w(KG_{n,k})=\left\lceil\frac{n}{k}\right\rceil$$
\end{theorem}

\begin{proof}
When $n=2k$, $KG_{2k,k}$ is bipartite, so $w(KG_{2k,k})=\chi(KG_{2k,k})=2$ (see \cite{BW1}). We thus assume $\left\lceil\frac{n}{k}\right\rceil=s+2$, for some $s\geq 1$.

Note that for any given vertex $A\subset\binom{[n]}{k}$, we  have  $N(A)=\binom{[n]\setminus A}{k}$. Since $sk<n-k\leq (s+1)k$, we can always choose a collection $\mathcal{C}$ of $(s+1)$ $k$-subsets of $[n]\setminus A$ whose union is $[n]\setminus A$, and so $\{A\}=\cap_{B\in\mathcal{C}} N(B)$. Thus the collection of singletons of vertices forms a $(s+1)$-stable family, and $w(KG_{n,k})\leq s+2$.  

If $s=1$, $KG_{n,k}$ is connected and not bipartite, so $w(KG_{n,k})=3$. From now on we focus on $s\geq 2$. We will show that $KG_{n,k}$ has property $\mathcal{P}^s_{2k}$, thus bounding $w(KG_{n,k})\geq s+2$. We start by establishing the following two claims.\\

    \textbf{Claim 1: } Let $0\leq r\leq k$, and $A_1,\dots, A_{s}$ be $k$-subsets of $[n]$ such that $|A_1\cap A_i|\geq r$ for $1\leq i\leq s$. Then $\displaystyle|\cup_{i=1}^{s}A_i|\leq sk-(s-1)r$.\\
    
    \textit{Proof of Claim 1:} Note $|A_i\setminus A_1|=|A_i|-|A_i\cap A_1|\leq k-r$. Thus%
    \begin{align*}
        \left| \bigcup_{i=1}^{s}A_i\right|&=|A_1|+\left|\bigcup_{i=2}^{s} (A_i\setminus A_1)\right| \leq k+\sum_{i=2}^{s}|A_i\setminus A_1|\\
        &\leq k+\sum_{i=2}^{s} (k-r)=sk-(s-1)r.
    \end{align*}\\[-3em]
    \phantom{x}\hfill$\triangle$\vspace{2em}

    Let $\T$ be an $s$-branching tree of depth 2. Label its vertices as: $$V(\T)=\{v_0,v_1,v_2,\dots,v_{s},v_{1}^1,\dots,v_1^s,v_2^1,\dots,v_2^s,\dots,v_s^1,\dots, v_s^s\},$$ where $v_0$ is the root, $\{v_1,\dots, v_s\}$ are $v_0$'s children and for each $1\leq i\leq s$ the children of $v_i$ are $\{v_i^1,\dots, v_i^s\}$ (Figure \ref{fig:tree_kneser}).\\
    
    \textbf{Claim 2: } Given an integer $0\leq r\leq (k-1)$, a fixed $(r+1)$-set $A_0\subset[n]$, and a function $f:L(\T)\to\binom{[n]}{k}$ such that $\displaystyle\left|\bigcap_{v_i^j\in L(\T)} f(v_i^j)\right|\geq r$. Then $f$ can be extended into a graph homomorphism $f:\T\to KG_{n,k}$ such that $A_0\subset f(v_0)$.\\
    
    \textit{Proof of Claim 2:} For each fixed $i$, the $k$-sets $f(v_i^1),\dots, f(v_i^s)$ satisfy the hypothesis of Claim 1, so $\left|\bigcup_{j=1}^{s} f(v_i^j)\right|\leq sk-(s-1)r$. 
    For each $1\leq i\leq s$, let $E_i=[n]\setminus\bigcup_{j=1}^{s} f(v_i^j)$, so if we assign any $k$-subset of $E_i$ as $f(v_i)$, it will be disjoint to $f(v_i^j)$ for all $v_i$'s children. Note $$|E_i|\geq n-sk+(s-1)r=k+\ell+(s-1) r\geq k+1+r,$$
    where $\ell:=n-(s+1)k.$ So $|E_i\setminus A_0|\geq |E_i|-|A_0|\geq k$. Assign as $f(v_i)$ any $k$-subset of $E_i\setminus A_0$, for each $1\leq i\leq s$. Now we define $f(v_0)$ disjoint to the $f(v_i)$'s. Note that $$\left|A_0\cup \bigcup_{i=1}^s f(v_i)\right|\leq |A_0|+\sum_{i=1}^s|f(s_i)|=(r+1)+sk$$
    
    So $[n]\setminus \left(A_0\cup \bigcup_{i=1}^s f(v_i)\right)$ has at least $n-((r+1)+sk)\geq k-(r+1)$ elements. Let $B$ be any $(k-r-1)$ such elements and assign $f(v_0):=A_0\uplus B$. Thus $f(v_0)$ is a $k$-set disjoint to all the sets $f(v_i)$, and so $f$ is a graph homomorphism as desired.$\triangle$\\

\begin{figure}
\begin{center}
\begin{subfigure}[b]{0.45\textwidth}
\centering
    \begin{tikzpicture}[x=0.8cm,y=0.8cm,line width=1pt]
\tikzstyle{every node}=[text=black,anchor=west]
\Large
\draw (0,2)--(-1.333,1);
\draw (0,2)--(1.333,1);
\draw[fill=blue] (0,2) circle (4pt) node[] {$v_0$};

\draw (-1.333,1)--(-0.667,0);
\draw (-1.333,1)--(-2,0);
\draw[fill=blue] (-1.333,1)  circle (4pt) node[anchor=east] {$v_1$};
\draw[fill=blue] (-0.667,0) circle (3pt) node[anchor=east] {$v_1^2$};
\draw[fill=blue] (-2,0) circle (3pt) node[anchor=east] {$v_1^1$};

\draw (1.333,1)--(0.667,0);
\draw (1.333,1)--(2,0);
\draw[fill=blue] (1.333,1)  circle (4pt) node[] {$v_2$};
\draw[fill=blue] (0.667,0) circle (3pt) node[] {$v_2^1$};
\draw[fill=blue] (2,0) circle (3pt) node[] {$v_2^2$};
\end{tikzpicture}
\caption{$\T$ in Theorem \ref{thm:kneser}.}\label{fig:tree_kneser}
\end{subfigure}
\begin{subfigure}[b]{0.45\textwidth}
\centering
    \begin{tikzpicture}[x=.8cm,y=.8cm,line width=1pt]
\tikzstyle{every node}=[text=black,anchor=west]
\Large
\draw (0,3)--(0,2);
\draw[fill=blue] (0,3) circle (4pt) node[] {$u$};

\draw (0,2)--(-1.333,1);
\draw (0,2)--(1.333,1);
\draw[fill=blue] (0,2) circle (4pt) node[] {$v_0$};

\draw (-1.333,1)--(-0.667,0);
\draw (-1.333,1)--(-2,0);
\draw[fill=blue] (-1.333,1)  circle (4pt) node[anchor=east] {$v_1$};
\draw[fill=blue] (-0.667,0) circle (3pt) node[anchor=east] {$v_1^2$};
\draw[fill=blue] (-2,0) circle (3pt) node[anchor=east] {$v_1^1$};

\draw (1.333,1)--(0.667,0);
\draw (1.333,1)--(2,0);
\draw[fill=blue] (1.333,1)  circle (4pt) node[] {$v_2$};
\draw[fill=blue] (0.667,0) circle (3pt) node[] {$v_2^1$};
\draw[fill=blue] (2,0) circle (3pt) node[] {$v_2^2$};
\end{tikzpicture}
\caption{$\T$ in Lemma \ref{lemma:connectivity}.}
\end{subfigure}
\caption{Labelings of tree $\T$ in the proofs of Theorem \ref{thm:kneser} and Lemma \ref{lemma:connectivity}, when $s=2$.}
\label{fig:label_trees}
\end{center}
\end{figure}

Let $T=T^s_{2k}$, labeled as in Figure \ref{fig:three_branching_map}. Fix $f:D(T)\to V(KG_{n,k})=\binom{[n]}{k}$, and denote $f(1)=\{a_1,\dots, a_k\}$. We will prove that $f$ is extendable to a graph homomorphism $\psi:T\to KG_{n,k}$.

 For any $k\geq 0$, denote by ${T}|_{k}$ its $k$-layer, that is, the vertices at distance $k$ to the root. And similarly, if $C=\{k_1,\dots, k_m\}\subset\mathbb{N}$, denote by $T|_C$ the subgraph induced by the vertices on layers $k_1,\dots, k_m$.  We will construct $\psi$ by induction over $r$, showing that if it is defined on $T_{(2k-2r+2)}$ it can be extended to $T|_{\{2k-2r,2k-2r+1,2k-2r+2\}}$ such that
$$\{a_1,\dots, a_r\}\subset\bigcap_{v\text{ in layer }(2k-2r)} \psi(v).$$

\begin{itemize}
    \item For $r=0$, define $\psi(v)=f(v)$ for all $v$ in layer $2k$, it trivially satisfies the claim.
    \item Suppose the claim is true for all values at most $r$, thus we need to extend $\psi$ to layers $C_r=\{2k-2r-2,2k-2r+1,2k-2r\}$. Note ${T}|_{C_r}$ is a collection of disjoint trees of depth 2, and let $\T$ be one of such trees. By the induction hypothesis, $\{a_1,\dots, a_r\}\subset \bigcap_{v\in\text{Leaves}(\T)} \psi(v)$. So we can apply Claim 2 to extend $\psi$ on $\T$, such that $\{a_1,\dots, a_r, a_{r+1}\}\subset\psi(\text{root of }\T)$. Doing this extension to all the trees in $T|_{C_r}$, we successfully extend $\psi$ to layers $C_r$, and since the roots of all such trees are in layer $(2k-2r-2)$, we have $$\{a_1,\dots, a_r,a_{r+1}\}\subset\bigcap_{v\text{ in layer }(2k-2r-2)} \psi(v).$$ 
\end{itemize}

Hence, when $r=k$, we get $\psi:T\to KG_{n,k}$ with $\{a_1,\dots, a_k\}\subset \psi(1)$, so $\psi(1)=f(1)$ as required. 
\end{proof}

We observe that $w(KG_{n,k})\leq\text{conn}\left[\mathcal{N}(KG_{n,k})\right]+3$, but the gap between the two quantities can be arbitrarily large, hence the strongest inequality that we could hope for is:

\begin{conjecture}\label{conj:conn}
$$ w(H) \le \textup{conn}[\mathcal{N}(H)] + 3.$$
\end{conjecture}


Inspired by a previous version of this paper where Conjecture \ref{conj:conn} first appeared, Dochtermann and Freij \cite{DFH18} proved it holds for some specific graphs with $\chi(G)>3$, in particular they proved the following:

\begin{theorem}[\cite{DFH18}]
If the first homology group $H_1(\mathcal{N}(G))$ contains an infinite cyclic subgroup, then $w(G)\leq 3$. And so,  if $G$ is also not bipartite, $$w(G)=3=\textup{conn}\left[\mathcal{N}(G)\right]+3\leq \chi(G). $$
\end{theorem}

This shows the conjecture may hold even if the connectivity of the neighborhood complex is far away from the chromatic number. Below, we show that the conjecture holds a.a.s. for $G(n,n^{-\alpha})$ for some specific ranges of $\alpha$. In particular, this provides examples where the conjecture holds even if warmth and connectivity of the neighborhood complex are close to each other and far away from the chromatic number. 

\begin{lemma}\label{lemma:connectivity}
If $G$ has property $\mathcal{P}^{s}_3$ then $\mathcal{N}(G)$ is $(s-1)$-connected. 
\end{lemma}

\begin{proof}[Proof]
Let $K=\mathcal{N}(G)$. Since $G$ has property $\mathcal{P}^s_3$, $K$ has a complete $(s-1)-$skeleton. Let $S$ be any set of $(s+1)$-vertices of $K$, and let $\partial S$ be the collection of $s$-subsets of $S$, thus $\partial S\subset K$. We will prove there is another vertex $w$ in $K$, such that the cone $w*\partial S\subset K$, which implies $\partial S$ is the boundary of a geometric $s$-ball in $K$. Since $S$ can be any $s$-face,  $K$ is $(s-1)$-connected.  

Let $\T$ be one branch of the $s$-branching tree of length 3, $T^s_3$. Label its vertices as follows: let $u$ be its root, $v_0$ be the only child of the root, $v_1, \dots, v_s$ the children of $v_0$ and $v_i^1,\dots, v_i^s$ the children of $v_i$, for each $1\leq i\leq s$ (Figure \ref{fig:label_trees}). 

Let $S=\{a_0, a_1,\dots, a_s\}$. Since $G$ is $s$-neighborly, there is a common neighbor of $S\setminus\{a_0\}$, call it $z$. Define $f:D(\T)\to G$ as follows
\begin{align*}
    f(u) &= z\\
    f(v_i^j) & = \begin{cases}
    a_j &\text{ if } i\neq j\\
    a_0 &\text{ if } i=j
    \end{cases}
\end{align*}

Since $G$ has property $\mathcal{P}^s_3$, $f$ can be extended into a graph homomorphism, in particular there is a vertex $w=f(v_0)$. 

Then $\{z,w\}\in E(G)$ and, by definition, $z$ is a common neighbor of $a_1,\dots,a_s$, thus $\{w,a_1,\dots,a_x\}$ is a face of $K$. 
At the same time, for each $1\leq i\leq s$, $$f(v_i)\in N(w)\cap\bigcap_{\substack{0\leq j\leq s\\ i\neq j}} N(a_j),$$%
so $\{w,a_0,\dots, a_s\}\setminus\{a_i\}$ is also a face of $K$, for each $1\leq i\leq s$. Therefore, $w*\partial S\subset K$ as desired.
\end{proof}

\begin{theorem}\label{thm:warmth_connectivity}
Let $p=n^{-\alpha}$ with $\alpha$ such that  $\frac{1}{s+1}<\alpha<\frac{1+s}{1+s+s^2}$, for some integer $s\geq 3$. Then a.a.s. $w(G(n,p))\leq \textup{conn}\left[\mathcal{N}(G(n,p))\right]+3$.
\end{theorem}

\begin{proof}[Proof of Theorem \ref{thm:warmth_connectivity}]
By Theorem \ref{sparse}, $w(G(n,p))\leq\lfloor1/\alpha +2\rfloor=s+2$.
Note $$\frac{1}{1+s+s^2}<1-s\alpha,$$ so choosing $\beta>\alpha$ if necessary, the argument in the proof of Corollary \ref{lemma:property_s_ell_dense} shows that a.a.s. $G(n,n^{-\beta})$ has property $\mathcal{P}^s_3$, and so a.a.s. $G(n,p)$ has property $\mathcal{P}^s_3$. Then by Lemma \ref{lemma:connectivity} $\textup{conn}\left[\mathcal{N}(G(n,p))\right]\geq s-1$, and the result follows.
\end{proof}

Even the following stronger version of might hold. This inequality would imply Conjecture \ref{LC}, since we already know that 
$$\textup{conn}[\mathcal{N}(H)] + 3 \le \chi(H).$$

\begin{conjecture}\label{conj:mob}
$$ m(H) \le \textup{conn}[\mathcal{N}(H)] + 3.$$
\end{conjecture}


We note Conjecture \ref{conj:mob} holds for complete graphs, since $m(K_n)=n$ (see \cite{BW1}) and $\text{conn}\left[\mathcal{N}(K_n)\right]=n-3$ (Theorem \ref{thm:kneser_conn} with $k=1$). Moreover, Theorems \ref{thm:kneser}, \ref{thm:kneser_conn} and \ref{ineqs} show it holds for Kneser graphs $KG_{n,k}$ with $k\geq 3$ and $(n,k)\neq(10,3)$. Brightwell and Winkler also proved in \cite{BW1}, that if $\chi(G)\leq3$, $w(G)=m(G)$, so it also holds for graphs with low chromatic number, in particular bipartite graphs and cycles. 

\section{Future directions}


We briefly suggest a few problems for future study.\\

\begin{itemize}
\item We conjectured in Section \ref{sec:topology} that
$$ w(H) \le \textup{conn}[\mathcal{N}(H)] + 3.$$
It would be interesting, though, to find any connection between the graph parameters coming from statistical physics and those coming from topology.\\

\item We conjecture that the property $w(H) \le k$ has a sharp threshold in the following sense, even though warmth is not a monotone graph property.

\begin{conjecture}
For each integer $d\geq 2$, there exists a constant $C_d$ such that for fixed $\epsilon > 0$ we have:
\begin{itemize}
    \item if $p \le \left(C_d - \epsilon\right) n^{-1/d}$, then a.a.s. $w\left(G(n,p)\right)\leq d+1$, and
    \item if $p \ge \left( C_d + \epsilon \right) n^{-1/d}$, then a.a.s. $w\left(G(n,p)\right)\geq d+2$. \\
\end{itemize}
\end{conjecture}

\item Brightwell and Winkler gave several examples of families of graphs where warmth and chromatic number are equal, namely bipartite graphs, complete graphs, and ``collapsible'' graphs. It might be interesting to find other families where warmth is close to chromatic number. Our work here indicates that Erd\H{o}s--R\'enyi random graphs will not work. ``Random Borsuk graphs'' taken by sampling $n$ uniform random points from a sphere $S^d$, and connecting nearly antipodal points, were studied in \cite{KM20}. For random Borsuk graphs, it seems likely to us that the connectivity of neighborhood complex is close to the chromatic number, but warmth might be small even when $d$ is large.

\end{itemize}

%

\section*{Acknowledgements}

We thank Persi Diaconis, Anton Dochtermann, and Peter Winkler for helpful and encouraging  conversations. Kahle and Martinez-Figueroa gratefully acknowledge support from NSF-DMS \#2005630.

\bibliography{wrefs}
\end{document}